\documentclass{amsproc}

\usepackage{amsmath}    
\usepackage{graphicx} 
\usepackage{verbatim} 
\usepackage{color} 
\usepackage{subfigure}  
\usepackage{hyperref}  
\usepackage{bbm}
\usepackage{dsfont} 
\usepackage{amssymb}
\usepackage{mathrsfs}
\usepackage[super]{nth}
\usepackage{enumitem,kantlipsum}
\usepackage{mathtools}

\newtheorem{theorem}{Theorem}
\newtheorem{proposition}{Proposition}
\newtheorem{lemma}{Lemma}

\newtheorem{defn}{Definition}
\newtheorem{remark}{Remark}
\newtheorem{question}{Question}

\newcommand{\Proj}{\mathbb{P}}
\renewcommand{\P}{\mathbb{P}}
\newcommand{\cC}{\mathcal{C}}

\newcommand{\E}{\mathbb{E}}

\newcommand{\II}{\mathbb{I}}
\newcommand{\LL}{\mathbb{L}}

\newcommand{\Z}{\mathbb{Z}}

\newcommand{\F}{\mathbb{F}}

\newcommand{\Gal}{\mathop{\rm Gal}}
\newcommand{\eqdef}{\mathop{=}^{\rm def}}

\renewcommand{\bf}[1]{\mathbf{#1}}

\newcommand{\es}[1]{\begin{equation}\begin{split}#1\end{split}\end{equation}}
\newcommand{\est}[1]{\begin{equation*}\begin{split}#1\end{split}\end{equation*}}

\newcommand{\Aut}{{\operatorname{Aut}}}
\newcommand{\wt}{{\operatorname{wt}}}
\newcommand{\PGL}{{\operatorname{PGL}}}
\newcommand{\GL}{{\operatorname{GL}}}
\newcommand{\SL}{{\operatorname{SL}}}

\newcommand{\legen}[2]{\left(\frac{#1}{#2}\right)}

\makeatletter
\let\@@pmod\pmod
\DeclareRobustCommand{\pmod}{\@ifstar\@pmods\@@pmod}
\def\@pmods#1{\mkern4mu({\operator@font mod}\mkern 6mu#1)}
\makeatother

\title[Weight Enumerators of Reed-Muller Codes from Cubic Curves]{Weight Enumerators of Reed-Muller Codes from Cubic Curves and their Duals}
\author{Nathan Kaplan}

\address{Nathan Kaplan -- Department of Mathematics, University of California\\ Irvine, CA  92697\\ \texttt{nckaplan@math.uci.edu} }

\begin{document}
\maketitle

\begin{abstract}
Let $\F_q$ be a finite field of characteristic not equal to $2$ or $3$.  We compute the weight enumerators of some projective and affine Reed-Muller codes of order $3$ over $\F_q$.  These weight enumerators answer enumerative questions about plane cubic curves.  We apply the MacWilliams theorem to give formulas for coefficients of the weight enumerator of the duals of these codes.  We see how traces of Hecke operators acting on spaces of cusp forms for $\SL_2(\Z)$ play a role in these formulas.
\end{abstract}

\section{Introduction}

Reed-Muller codes are some of the most famous and well studied examples of \emph{evaluation codes}.  Let $\mathcal{P} = \{P_1,\ldots, P_N\}$ be a subset of points of the affine space $\F_q^n$ and let $\mathcal{V}$ be a finite subspace of polynomials in $\F_q[x_1,\ldots, x_n]$.  Consider the evaluation map:
\begin{eqnarray*}
\operatorname{ev}_{\mathcal{P}} \colon \mathcal{V} & \mapsto & \F_q^N \\
f & \mapsto & \left(f(P_1),\ldots, f(P_N)\right).
\end{eqnarray*}
The \emph{evaluation code} $\operatorname{ev}(\mathcal{V},\mathcal{P})$ is the image of this map.  The  \emph{affine Reed-Muller code} of order $k$ and length $q^n$, denoted $\operatorname{RM}^A_q(k,n)$, comes from choosing $\mathcal{V}$ to be $\F_q[x_1,\ldots, x_n]_{\le k}$, the vector space of polynomials of degree at most $k$, and $\mathcal{P}$ to be the set of all points in the affine space $\F_q^n$.  When the evaluation map is injective, which is certainly the case for $k<q,\ \operatorname{RM}^A_q(k,n) \subset \F_q^{q^n}$ is a $\binom{n+k}{k}$-dimensional linear code.

Let $\mathcal{P} = \{P_1,\ldots, P_N\}$ be a subset of points in the projective space $\Proj^n(\F_q)$.  It does not make sense to evaluate an element of $\F_q[x_0,\ldots, x_n]$ at a projective point, so we make a choice of affine representative for each, giving a set  $\mathcal{P}' = \{P'_1,\ldots, P'_N\}$ with each $P'_i \in \F_q^{n+1}$.  Let $\mathcal{V}$ be a subspace of $\F_q[x_0,\ldots, x_n]_k$, the set of homogeneous polynomials of degree $k$ (including the zero polynomial).  We now define an evaluation code as in the previous paragraph.  When $\mathcal{P}$ consists of all points in $\Proj^n(\F_q)$ and $\mathcal{V}$ is all of $\F_q[x_0,\ldots, x_n]_k$, this construction defines the \emph{projective Reed-Muller code} of order $k$ and length $N = |\Proj^n(\F_q)| = (q^{n+1}-1)/(q-1)$, denoted $\operatorname{RM}^P_q(k,n)$.  When the evaluation map is injective, which is certainly the case when $k < q,\ \operatorname{RM}^P_q(k,n) \subseteq \F_q^{N}$, is a $\binom{n+k}{k}$-dimensional linear code.  In this paper we write $C_{n,k}$ for $\operatorname{RM}^P_q(k,n)$ and  $C^A_{n,k}$ for $\operatorname{RM}^A_q(k,n)$.

\begin{remark}
\begin{enumerate}[wide, labelwidth=!, labelindent=0pt]  
\item This definition of Reed-Muller codes depends on an ordering of the points, and in the projective case, on a choice of affine representatives.  These Reed-Muller codes are not uniquely defined, but satisfy a strong form of equivalence \cite[Section 1.7]{HP}.  

An $n \times n$ \emph{monomial matrix} is a matrix of the form $M = DP$ where $D$ is an $n \times n$ invertible diagonal matrix and $P$ is a permutation matrix.  A linear code $C$ with generator matrix $G$ is \emph{monomially equivalent} to a linear code $C'$ of the same length $n$ if there exists an $n \times n$ monomial matrix $M$ such that $G' = GM$ is a generator matrix for $C'$. Informally, two codes are monomially equivalent if you can get from one to the other by permuting coordinates and then scaling each coordinate by an element of $\F_q^*$.  The Reed-Muller codes introduced above are not uniquely defined, but are well defined up to monomial equivalence.

\item The definition of the projective Reed-Muller code given by Lachaud in \cite{Lachaud} is phrased differently, but corresponds to making the standard choice where the affine representative $(x_0',\ldots, x_n')$ for $[x_0:\cdots: x_n]$ satisfies $x_i' = 1$ for the smallest $i$ such that $x_i \neq 0$.

\end{enumerate}
\end{remark}

Let $C \subseteq \F_q^N$ be a code and choose a subset $S \subseteq \{1,2,\ldots, N\}$ of size $m$.  The \emph{punctured code} $C' \subseteq \F_q^{N-m}$ comes from taking each codeword of $C$ and erasing the coordinates in positions in $S$.  The affine Reed-Muller code $C_{n,k}^A$ is the projective Reed-Muller code $C_{n,k}$ punctured at the positions corresponding to the $\F_q$-points of any $\F_q$-rational hyperplane.  Since $\operatorname{PGL}_{n+1}(\F_q)$ acts transitively on  hyperplanes, the corresponding punctured codes are monomially equivalent.  

Properties of $C_{n,k}$ and $C_{n,k}^A$ answer questions about the set of all  degree $k$ hypersurfaces in $n$-dimensional affine and projective space over $\F_q$.
\begin{defn}
For elements of $\F_q^N,\ x = (x_1,\ldots, x_N)$ and $y= (y_1,\ldots, y_N)$, the \emph{Hamming distance} of $x$ and $y$ is 
\[
d(x,y) =  \#\{i\ |\ x_i \neq y_i \}.
\]

The \emph{Hamming weight} of $x$ is its Hamming distance from the all zero vector,
\[
\wt(x) = d(x,\bf{0}) = \#\{i\ |\ x_i \neq 0\}.
\]

The Hamming weight enumerator of a code $C \subseteq \F_q^N$ is a homogeneous polynomial in two variables that keeps track of the number of codewords of $C$ of each weight.  More formally,
\[
W_C(X,Y) = \sum_{c\in C} X^{N-\wt(c)} Y^{\wt(c)} = \sum_{i=0}^N A_i X^{N-i} Y^i,
\]
where $A_i = \#\{c\in C\ |\ \wt(c) = i\}$.
\end{defn}
We see that the following questions are equivalent.
\begin{question}
Let $N = |\Proj^n(\F_q)| = (q^{n+1}-1)/(q-1)$.  Suppose that $k<q$, so $C_{n,k} \subseteq\F_q^N$ is a $\binom{n+k}{k}$-dimensional linear code.
\begin{enumerate}
\item How many homogeneous polynomials of degree $k,\ f \in \F_q[x_0,\ldots, x_n]_k$ are such that the hypersurface $\{f = 0\} \subseteq \Proj^n$ has exactly $N-i\ \F_q$-rational points?

\item What is the weight $i$ coefficient $A_i$ of $W_{C_{n,k}}(X,Y)$?
\end{enumerate}
\end{question}
\noindent We can ask analogous questions for the affine Reed-Muller codes.

Many authors have studied minimum distances and other invariants of affine and projective Reed-Muller codes.  For example, see \cite{DGM, vanderGeerSchoofvanderVlugtC, HeijnenPellikaan, Lachaud, Sorensen}.  There are not many examples for which the weight enumerators of Reed-Muller codes have been computed explicitly.  For $k=1$ these codes come from hyperplanes and their weight enumerators are easy to compute.  See \cite{Jurrius} for more refined information on these codes.  Aubry considers codes from projective quadric hypersurfaces in \cite{Aubry}. Elkies computes the weight enumerators of the duals of these codes in the unpublished preprint \cite{Elkies}.  Elkies also computes the weight enumerator of the code of projective plane cubic curves $C_{2,3}$ and the weight enumerator of the code of cubic surfaces $C_{3,3}$.  Knowledge of $W_{C_{2,3}}(X,Y)$ plays an important role in the analogous computation for $C_{3,3}$ since cones over plane cubic curves arise as singular cubic hypersurfaces in $\Proj^3$.  For $n=1$ and any $k$, affine Reed-Muller codes are Reed-Solomon codes and projective Reed-Muller codes are (doubly) extended, or projective, Reed-Solomon codes.  In these cases the weight enumerators are well understood.

Results for small $k$ lead to corresponding results for $k$ large by considering  dual codes.
\begin{defn}
For $x,y \in \F_q^N,\ x=(x_1,\ldots, x_N)$ and $y=(y_1,\ldots, y_N)$, let
\[
\langle x,y\rangle = \sum_{i=1}^N x_i y_i \in \F_q.
\]

The \emph{dual code} $C^\perp$ of a linear code $C\subset \F_q^N$ is defined by 
\[
\{y \in \F_q^N\ |\ \langle x,y\rangle = 0\ \forall x\in C\}.
\]
\end{defn}
\noindent The dual of a Reed-Muller code is also a Reed-Muller code.  For details, see \cite[Theorem 6.11.3]{vanLint}.  

The MacWilliams theorem states that the weight enumerator of a linear code $C$ determines the weight enumerator of $C^{\perp}$.
\begin{theorem}[MacWilliams]\label{MacThm}
Let $C \subseteq \F_q^N$ be a linear code.  Then 
\[
W_{C^\perp}(X,Y) = \frac{1}{|C|} W_C(X+(q-1)Y,X-Y).
\]
\end{theorem}
\noindent For a proof, see for example \cite[Theorem 3.5.3]{vanLint}.

Theorem \ref{MacThm} implies that we can give an expression for $W_{C_{n,n-k-1}}(X,Y)$ in terms of $W_{C_{n,k}}(X,Y)$. However, this does not make it clear what sorts of inputs are necessary to give formulas for the coefficients of $W_{C_{n,n-k-1}}(X,Y)$.  Applying the MacWilliams theorem in this way, Elkies gives closed formulas for the coefficients of $W_{C_{n,n-3}}(X,Y)$ \cite{Elkies}.

We have two main goals in this paper:
\begin{enumerate}[wide, labelwidth=!, labelindent=0pt]  
\item Compute $W_{C_{2,3}}(X,Y)$ following the strategy of Elkies \cite{Elkies}, and use it to compute the significantly more complicated $W_{C^A_{2,3}}(X,Y)$.

\item Apply the MacWilliams theorem to give formulas for low-weight coefficients of $W_{C_{2,3}^{\perp}}(X,Y)$ and $W_{(C_{2,3}^{A})^{\perp}}(X,Y)$.  Computing these expressions leads to evaluating sums that have been considered by Birch \cite{Birch} and Ihara \cite{Ihara} and gives a connection between traces of Hecke operators acting on spaces of cusp forms for $\SL_2(\Z)$ and coefficients of weight enumerators of Reed-Muller codes.

\end{enumerate}

\subsection{Comparison with Previous Work}\ \\

The results of this paper are similar in spirit to those of the author's paper with Petrow \cite{KaplanPetrow1}.  In that paper, we study a refined weight enumerator of $C_{1,4}$, the \emph{quadratic residue weight enumerator},  that keeps track of not just the number of coordinates of a codeword that are zero or nonzero, but also the number of nonzero coordinates that are squares in $\F_q^*$.  This leads to statements about the number of homogeneous quartic polynomials $f \in \F_q[x,z]_4$ for which $y^2 = f(x,z)$ is an elliptic curve with a given number of $\F_q$-points and for which $f(x,z)$ has a specified number of $\F_q$-rational roots.  Formulas for low-weight coefficients of the quadratic residue weight enumerator of $C_{1,4}^{\perp}$ involve traces of Hecke operators acting on spaces of cusp forms for the congruence subgroups $\SL_2(\Z), \Gamma_0(2)$, and $\Gamma_0(4)$.  In this paper we see that elliptic curves with prescribed $3$-torsion play a role in computing $W_{C_{2,3}^A}(X,Y)$.  If we isolate the contribution to $W_{C_{2,3}^\perp}(X,Y)$ from elliptic curves with specified $3$-torsion, we get formulas involving traces of Hecke operators acting on spaces of cusp forms for the congruence subgroups $\SL_2(\Z), \Gamma_0(3)$, and $\Gamma(3)$.

A smooth projective plane cubic curve $C$ over $\F_q$ has genus $1$ and every such curve has an $\F_q$-rational point, so $C$ defines an elliptic curve.  Several authors have studied families of codes coming from elliptic curves over finite fields where traces of Hecke operators play a role in formulas for coefficients of the weight enumerators.  For example, weight enumerators of Zetterberg and Melas codes are studied in \cite{vanderGeerSchoofvanderVlugtB, SchoofvanderVlugt}, and the Eichler-Selberg trace formula for $\Gamma_1(4)$ plays an important role in the proofs.  In \cite{vanderGeerSchoofvanderVlugtA}, the authors study families of codes related to supersingular elliptic curves in characteristic $2$ and to certain Reed-Muller codes.  Other results of this type are described in the survey of Schoof \cite{SchoofSurvey}.

In our study of affine plane cubic curves, we see that $3$-torsion of elliptic curves over $\F_q$ plays an important role.  We see formulas that are reminiscent of Schoof's formulas for the number of projective equivalence classes of plane cubic curves \cite{Schoof}.

\subsection{Outline of the Paper}\ \\

Our strategy is to separate the weight enumerators into the contribution from singular cubics and the contribution from smooth cubics.  That is, we write
\[
W_{C_{2,3}}(X,Y) = W^{\text{singular}}_{C_{2,3}}(X,Y) + W^{\text{smooth}}_{C_{2,3}}(X,Y).
\]
Applying the MacWilliams theorem we see that for $q \ge 3$,
\[
q^{10} W_{C^{\perp}_{2,3}}(X,Y) = W^{\text{singular}}_{C_{2,3}}(X+(q-1)Y,X-Y) + W^{\text{smooth}}_{C_{2,3}}(X+(q-1)Y,X-Y).
\]
We compute the contribution to the weight enumerator of $C_{2,3}^\perp$ from singular cubics and from smooth cubics separately.  We follow a similar strategy for affine cubics, but the details are more complicated.

In the next section we recall the rational point count distribution for singular plane cubic curves in $\Proj^2(\F_q)$.  In Section 3, we recall some results about elliptic curves over finite fields and compute $W^{\text{smooth}}_{C_{2,3}}(X,Y)$.  In Section \ref{LowWeightProj}, we combine formulas of Birch and Ihara with the MacWilliams theorem to give formulas for low-weight coefficients of $W_{C_{2,3}^{\perp}}(X,Y)$.  We analyze singular affine cubics in Section \ref{SingularAffineCubics}, and smooth affine cubics in Section \ref{AffineSmooth}.  This involves a discussion of rational inflection points on smooth cubic curves.  In the final section, we discuss formulas for low-weight coefficients of $W_{(C_{2,3}^A)^{\perp}}(X,Y)$.  We recall some generalizations of the results of Birch and Ihara due to the author and Petrow that allow us to understand the contribution to these weight enumerators from elliptic curves with prescribed $3$-torsion.

\begin{remark}
\begin{enumerate}[wide, labelwidth=!, labelindent=0pt]  

\item We have done extensive calculations in the computer algebra system Sage verifying the results of this paper.  It would not have been possible to write down the formulas we obtain here, see for example Theorem \ref{LowWeightDual}, without explicit computer verification.

\item In Section \ref{SingularAffineCubics} there are issues related to singular irreducible cubic curves that arise in the computation of $W_{C_{2,3}^A}(X,Y)$ that are different in characteristic $3$.  It is likely that with some additional effort, this case can be addressed.   While some behavior for cubic curves in characteristic $2$ is special \cite[Chapter 11]{Hirschfeld}, our results seem to hold here.  Since we have not verified our results computationally in this case, we assume that the characteristic of $\F_q$ is not $2$ in the statement of our main theorems.

\end{enumerate}

\end{remark}

\section{Singular Projective Plane Cubic Curves}

There is a short list of isomorphism classes of singular projective plane cubic curves in $\Proj^2(\F_q)$.  For an extensive discussion of the classification of such curves, including normal forms for projective equivalence classes, see \cite[Chapter 11]{Hirschfeld}.  We recall the following chart from the paper of Elkies \cite{Elkies}, which is related to \cite[Theorem 11.3.11]{Hirschfeld}

\begin{lemma}\label{SingularProjectiveCubicList}
Every homogeneous cubic $f \in \F_q[x_1, x_2, x_3]_3$ such that $\{f=0\}$ is not a smooth plane cubic is one of the $15$ types listed in the following table together with the number of $f\in C_{2,3}$ of that type and the weight of every such $f$.
\begin{center}
 \begin{tabular}{@{} | p{6.4cm} |  p{4.3cm}  |  p{1.6cm} |} 
\hline
Shape of $\{f=0\}$ & number of such $f \in C_{2,3}$ & $\wt(f)$ \\
\hline\hline
$\Proj^2(\F_q)$ & $1$ & $0$\\
triple line & $q^3-1$ & $q^2$\\
line and double line & $(q^3-1)(q^2+q)$ & $q^2-q$\\
$3$ concurrent lines: & & \\
\ \ \ all rational & $(q^3-1)(q^3-q)/6$ & $q^2-2q$ \\
\ \ \ $1$ rational, $2$ conjugate by $\Gal(\F_{q^2}/\F_q)$ & $(q^3-1)(q^3-q)/2$ & $q^2$ \\
\ \ \ $3$ conjugate by $\Gal(\F_{q^3}/\F_q)$ & $(q^3-1)(q^3-q)/3$ & $q^2+q$ \\
$3$ non-concurrent lines: & & \\
\ \ \ all rational & $(q^3-1)(q^4+q^3)/6$ & $(q-1)^2$ \\
\ \ \ $1$ rational, $2$ conjugate by $\Gal(\F_{q^2}/\F_q)$ & $(q^3-1)(q^4-q^3)/2$ & $q^2-1$ \\
\ \ \ $3$ conjugate by $\Gal(\F_{q^3}/\F_q)$ & $(q-1)^2 (q^5-q^3)/3$ & $q^2+q+1$ \\
conic and tangent line & $(q^5-q^2)(q^2-1)$ & $q^2-q$\\
conic and line meeting: & & \\
\ \ \ in $2$ rational points & $(q^6-q^3)(q^2-1)/2$ & $q^2-q+1$ \\
\ \ \ in $2$ conjugate points & $(q^6-q^3)(q-1)^2/2$ & $q^2-q-1$ \\
cubic with a cusp & $(q^3-1)(q^3-q)q^2$ & $q^2$ \\
cubic with a node: & & \\
\ \ \ with rational slopes & $(q^3-1)(q^3-q)(q^3-q^2)/2$ & $q^2+1$ \\
\ \ \ with slopes conjugate by $\Gal(\F_{q^2}/\F_q)$ & $(q^3-1)(q^3-q)(q^3-q^2)/2$ & $q^2-1$ \\
\hline
TOTAL & $q^9+q^8-q^6-q^5+q^4$ & \\
\hline
\end{tabular}
\end{center}
\end{lemma}

We can state this result in terms of the contribution to $W_{C_{2,3}}(X,Y)$ from homogeneous cubic polynomials $f$ for which $\{f=0\}$ is not a smooth cubic curve.  
\begin{lemma}
We have 
\begin{eqnarray*}
& & W^{\text{sing}}_{C_{2,3}}(X,Y)  =  X^{q^2+q+1} + 
\frac{(q^3-1)(q^3-q)}{6} X^{3q+1} Y^{q^2-2q} \\
& + &
\frac{(q^3-1) (q^4+q^3)}{6}  X^{3q} Y^{q^2-2q+1} + 
\frac{(q^3-1)(q^3-q^2)(q^2-q)}{2} X^{2q+2} Y^{q^2-q-1} \\
& + &
(q^3-1)(q^2+q)(q^2-q+1) X^{2q+1} Y^{q^2-q} 
 +  
\frac{(q^6-q^3)(q^2-1) }{2} X^{2q} Y^{q^2-q+1} \\
& + &
\frac{(q^3-1)(q^6-q^5)}{2}  X^{q+2} Y^{q^2-1} 
 + 
\frac{(q^3-1)(2q^5-q^3-q+2)}{2} X^{q+1} Y^{q^2} \\
  & + &
\frac{(q^3-1)(q^3-q)(q^3-q^2)}{2} X^q Y^{q^2+1} 
 +
\frac{(q^3-1)(q^3-q)}{3} X Y^{q^2+q} \\
 & +  &
\frac{(q-1)(q^3-q) (q^3-q^2)}{3} Y^{q^2+q+1}.
\end{eqnarray*}
\end{lemma}

The coefficients of $W^{\text{sing}}_{C_{2,3}}(X,Y)$ are polynomials in $q$, so for each $j$, the contribution to the $X^{q^2+q+1-j} Y^j$ coefficient of $W_{C_{2,3}^{\perp}}(X,Y)$ from these singular cubics is a polynomial in $q$.  For example, the $X^{q^2+q} Y$ coefficient of $W^{\text{sing}}_{C_{2,3}}(X+(q-1)Y,X-Y)$ is $(q^3-1)(q^3-q)(q^4-q^3)$.

\section{Smooth Projective Plane Cubic Curves}\label{SmoothProjectiveCubic}

In order to determine $W^{\text{smooth}}_{C_{2,3}}(X,Y)$ we answer two questions.
\begin{question}\label{Q2}
\begin{enumerate}[wide, labelwidth=!, labelindent=0pt] 
\item Given an isomorphism class of an elliptic curve $E$ defined over $\F_q$, how many elements of $C_{2,3}$ define a smooth plane cubic isomorphic to $E$?

\item How many isomorphism classes of an elliptic curve $E/\F_q$ have $\#E(\F_q) =~{q+1-t}$?
\end{enumerate}
\end{question}

The first of these questions is answered by Elkies in \cite{Elkies}.
\begin{lemma}\label{NumberCubics}
For every elliptic curve $E/\F_q$, the number of polynomials ${f \in\F_q[x,y,z]_3}$ such that the zero-locus $\{ f=0 \}$ is isomorphic to $E$ equals $\#\GL_3(\F_q)/\#\Aut_{\F_q}(E)$.
\end{lemma}
We give the proof from \cite{Elkies} for completeness, including some of the presentation in \cite[Section 5]{Schoof}. This argument will play a role in our analysis of smooth affine cubics.

\begin{proof}
The number of polynomials $f \in \F_q[x,y,z]_3$ such that $\{f = 0\} \cong E$ is $q-1$ times the number of smooth projective plane cubic curves $C$ with $C \cong E$.  

Let $i \colon E \hookrightarrow \P^2$ be a closed immersion defined over $\F_q$, with image $i(E) = C$. There is a one-to-one correspondence between $E(\F_q)$ and $C(\F_q)$.  The sheaf $i^*\mathcal{O}(1)$ is a very ample invertible sheaf $\mathcal{L}(D)$. Here $D$ is a divisor of degree $3$ defined over $\F_q$, and only its class is determined by $i$.

An isomorphism between $C$ and $E$ requires a choice of a point $P_0 \in C(\F_q)$ that is sent to the identity $O$ of the group law on $E$.  There are $\#E(\F_q) = q+1-t$ choices for $P_0$.  

We claim that the number of pairs $(C,P_0)$ where $C$ is a smooth projective plane cubic defined over $\F_q$ with $C \cong E$ and $P_0 \in C(\F_q)$ is $(q+1-t) \#\PGL_3(\F_q)/\#\Aut_{\F_q}(E)$.  Proving this claim completes the proof of the lemma.  A plane cubic curve $C$ gives not only a genus one curve $E$, but a degree $3$ divisor class on $E$ and a choice of basis for the global sections of that divisor class $(x,y,z)$ up to $\F_q^*$ scaling.  

Starting from $E$, there are $\#E(\F_q) = q+1-t$ divisor classes of $E$ of degree $3$ defined over $\F_q$.  We claim that for each one, we get $\#\PGL_3(\F_q)/\#\Aut_{\F_q}(E)$ embeddings $i \colon  E \hookrightarrow \P^2$ defined over $\F_q$.  This will complete the proof.

By Riemann-Roch, a degree $3$ divisor has a $3$-dimensional space of global sections.  Choosing a basis for the space of sections gives a closed immersion $i \colon  E \hookrightarrow \P^2$ defined over $\F_q$.  Two choices yield the same embedding, including the same image $P_0$ of the identity element $O \in E(\F_q)$, if and only if they are related by an automorphism of $E$ and an $\F_q^*$ scaling.

\end{proof}

We turn to the second part of Question \ref{Q2}, following the presentation in \cite{KaplanPetrow2}.  Let $E$ be an elliptic curve defined over $\F_q$. When we mention an elliptic curve $E$ we always implicitly mean the isomorphism class of $E$.  With this convention in mind, let $\cC=\{E/\F_q\}$, the set of $\F_q$-isomorphism classes of elliptic curves defined over $\F_q$.  We have
\[
\sum_{E\in \cC} \frac{1}{\#\Aut_{\F_q}(E)} = q,
\]
so the finite set $\cC$ is a probability space where a singleton $\{E\}$ occurs with~probability
\[
\P_q(\{E\}) = \frac{1}{q \# \Aut_{\F_q}(E)}.
\]  
Let $t_E\in \Z$ denote the trace of the Frobenius endomorphism associated to $E$.  We have $t_E=q+1 -\#E(\F_q)$ and by Hasse's theorem $t^2_E \leq 4q$.  For an integer $t$, let $\cC(t)$ be the subset of $\cC$ for which $t_E = t$.  Using this terminology, Lemma \ref{NumberCubics} gives an expression for $W_{C_{2,3}}^{\text{smooth}}(X,Y)$.

\begin{proposition}\label{WE_C23_Smooth}
Let $q \ge 3$.  Then 
\[
W_{C_{2,3}}^{\text{smooth}}(X,Y) = (q^3-1)(q^3-q)(q^3-q^2) q \sum_{t^2 \le 4q} \P_q(\cC(t)) X^{q+1-t} Y^{q^2+t}.
\]
\end{proposition}

We recall results due to Deuring \cite{Deuring}, Waterhouse \cite{Waterhouse}, and Schoof \cite{Schoof}, that express $\P_q(\cC(t))$ in terms of class numbers of orders in imaginary quadratic fields.  For $d<0$ with $d \equiv 0,1\pmod{4}$, let $h(d)$ be the class number of the unique quadratic order of discriminant $d$.  Let 
\es{
h_w(d)\eqdef 
\begin{cases} h(d)/3, &\text{ if } d = -3, \\ 
h(d)/2, & \text{ if } d=-4, \\ 
h(d) & \text{ if } d < 0,\ d \equiv 0,1 \pmod*{4}, \text{ and } d\neq -3,-4, \\
0 & \text{otherwise}, 
\end{cases}
} 
and for $\Delta \equiv 0,1\pmod 4$ let
\es{
\label{HKcn}H(\Delta) \eqdef \sum_{d^2 \mid \Delta} h_w\left(\frac{\Delta}{d^2}\right)
} 
be the \emph{Hurwitz-Kronecker class number}. For $a\in \Z$ and $n$ a positive integer, the \emph{Kronecker symbol} $\Big(\frac{a}{n}\Big)$ is defined to be the completely multiplicative function of $n$ such that if $p$ is an odd prime $\legen{a}{p}$ is the quadratic residue symbol, and if $p=2$,
\es{
\left(\frac{a}{2}\right) \eqdef \begin{cases} 0 & \text{ if } 2 \mid a, \\ 
1 & \text{ if } a \equiv \pm 1 \pmod 8, \\ -1 & \text{ if } a \equiv \pm 3 \pmod 8 .
\end{cases} 
}
The following is a weighted version of \cite[Theorem 4.6]{Schoof}.
\begin{lemma}\label{S46withweights}
Let $t\in \Z$.  Suppose $q = p^v$ where $p$ is prime and $v\geq 1$.  Then if $q$ is not a square
\begin{alignat*}{3}
\P_q(\cC(t))  = &  \frac{1}{2q}H(t^2-4q) \quad && \text{ if } t^2 < 4q \text{ and } p\nmid t,\\
 = & \frac{1}{2q}H(-4p)\quad  && \text{ if } t=0, \\
  = & \frac{1}{4q} \quad  && \text{ if } t^2=2q \text{ and } p =2, \\
 = & \frac{1}{6q} && \text{ if } t^2=3q \text{ and } p =3, 
\end{alignat*}
 and if $q$ is a square
\begin{alignat*}{3}
\P_q(\cC(t))  = &  \frac{1}{2q}H(t^2-4q)\quad  && \text{ if } t^2 < 4q \text{ and } p\nmid t,\\
 = & \frac{1}{4q}\left(1 - \legen{-4}{p}\right) \quad   && \text{ if } t=0, \\
 = & \frac{1}{6q}\left(1 - \legen{-3}{p}\right) \quad  && \text{ if } t^2 = q, \\
  = & \frac{p-1}{24q}\quad   && \text{ if } t^2 = 4q,
\end{alignat*}
 and $\P_q(\cC(t)) = 0$ in all other cases.
\end{lemma}

Combining Proposition \ref{WE_C23_Smooth} and Lemma \ref{S46withweights} gives an expression for $W_{C_{2,3}}^{\text{smooth}}(X,Y)$ in terms of class numbers of orders in imaginary quadratic fields.

We close this section with a result about $\F_q$-rational inflection lines of smooth plane projective cubic curves that we apply in Section \ref{AffineSmooth} when we compute $W^{\text{smooth}}_{C_{2,3}^A}(X,Y)$.  
\begin{defn}
A non-singular point $P$ of an absolutely irreducible cubic curve $C$ is an \emph{inflection point} if the tangent line $L$ at $P$ has contact order $3$ with $C$.  In this case $L$ is called an \emph{inflection line}.  In particular, $L$ does not intersect any other points of $C$.  
\end{defn}
Let $\II(C)$ denote the number of $\F_q$-rational inflection lines of $C$.  For an extensive discussion the possibilities for $\II(C)$, see \cite[Chapter 11]{Hirschfeld}.

\begin{proposition}\label{CountInflectionSmooth}
Let $E$ be an elliptic curve defined over $\F_q$.  
\begin{enumerate}
\item If $E(\F_q)[3]$ is trivial, then every smooth projective plane cubic $C$ with $C \cong E$ has $\II(C) = 1$.

\item If $E(\F_q)[3] \cong \Z/3\Z$, then $1/3$ of the smooth projective plane cubics $C$ with \\
$C \cong E$ have $\II(C) = 3$, and the remaining $2/3$ have $\II(C) = 0$.

\item If $E(\F_q)[3] \cong \Z/3\Z \times \Z/3\Z$, then $1/9$ of the smooth projective plane cubics $C$ with $C  \cong E$ have $\II(C) = 9$, and the remaining $8/9$ have $\II(C) = 0$.
\end{enumerate}
\end{proposition}

We recall some material from \cite[Section 4]{BruenHirschfeldWehlau}, which cites \cite{Lang}, to give a description of the inflection points in terms of the geometry of $C$.  Let $P * Q$ denote the third intersection of the line $PQ$ with $C$.  When $Q=P$, the line $PQ$ is the tangent at $P$ and $P*P = P_t$ is the \emph{tangential} of $P$.  If $P$ is an inflection point, $P_t = P$.  Choose a point $O \in C(\F_q)$ to be the identity of the group law on $C(\F_q)$ and recall that the group operation is defined by
\[
P \oplus Q = (P*Q)*O.
\]
Let $N$ denote $O_t$, the tangential of $O$.   Then $P$ is an inflection point if and only if $3P = N$.

We give a proof of Proposition \ref{CountInflectionSmooth} similar to the proof of \cite[Lemma 3]{KaplanPetrow1}.
\begin{proof}
Recall from the proof of Lemma \ref{NumberCubics} that $C$ defines not only an elliptic curve $E$, but also a degree $3$ divisor of $E$ defined over $\F_q$.  There is a one-to-one correspondence between $C(\F_q)$ and $E(\F_q)$.  As we vary over all cubics $C$ with $C \cong E$ we get each degree $3$ divisor class of $E$ defined over $\F_q$ the same number of times.  

This divisor is linearly equivalent to a unique one of the form $2 O + P$ where $O$ is the identity element of the group law on $E$ and $P \in E(\F_q)$.  A point $Q \in E(\overline{\F}_q)$ gives an inflection point of the cubic if and only if $3Q \sim 2O + P$, or equivalently, $3Q = P$ in the group law on $E$.

We vary over all choices of $P$ and consider how many $Q$ occur as points with $3Q = P$.  
\begin{enumerate}[wide, labelwidth=!, labelindent=0pt]  
\item If $\#E(\F_q) \not\equiv 0 \pmod{3}$, then the map $Q \mapsto 3Q$ is an isomorphism and every $P$ gives exactly one such $Q$.  

\item If $\#E(\F_q) \equiv 0 \pmod{3}$ then there are two possibilities for the group structure of $E(\F_q)[3]$.  

\begin{enumerate} 
\item If $E(\F_q)[3] \cong \Z/3\Z$, then $2/3$ of the points of $E(\F_q)$ have $0$ preimages under the map $Q \mapsto 3Q$, and $1/3$ have exactly $3$.  

\item If $E(\F_q)[3] \cong \Z/3\Z \times \Z/3\Z$ then $8/9$ of the points of $E(\F_q)$ have $0$ preimages under the map $Q \mapsto 3Q$, and $1/9$ have exactly $9$.
\end{enumerate}
\end{enumerate}
\end{proof}

\section{Low-Weight Coefficients of $W_{C_{2,3}^\perp}(X,Y)$}\label{LowWeightProj}

In this section we give formulas for low-weight coefficients of the weight enumerator of $C_{2,3}^\perp$.  By the MacWilliams theorem, we have 
\[
q^{10} W_{C_{2,3}^\perp}(X,Y) = W_{C_{2,3}}^{\text{singular}}(X+(q-1),X-Y) + W_{C_{2,3}}^{\text{smooth}}(X+(q-1)Y,X-Y).
\]
Since the coefficients of $W_{C_{2,3}}^{\text{singular}}(X,Y)$ are polynomials in $q$, the $X^{q^2+q+1-j} Y^j$ coefficient of $W_{C_{2,3}}^{\text{singular}}(X+(q-1)Y,X-Y)$ is a polynomial in $q$ for any $j$.  For the rest of this section we focus on $W_{C_{2,3}}^{\text{smooth}}(X+(q-1)Y,X-Y)$.

Applying the binomial theorem, we see that the $X^{q^2+q+1-j} Y^j$ coefficient of  
\begin{eqnarray*}
& & W_{C_{2,3}}^{\text{smooth}}(X+(q-1)Y,X-Y)  =  \\
& &  (q^3-1)(q^3-q)(q^3-q^2) q \sum_{t^2 \le 4q} \P_q(\cC(t)) (X+(q-1)Y)^{q+1-t} (X-Y)^{q^2+t}
\end{eqnarray*}
is a linear combination of expressions 
\[
\sum_{t^2 \le 4q}  \P_q(\cC(t)) t^{k}
\] 
for $k \in \{0,1,\ldots, j\}$, with coefficients that are polynomials in $q$.  Since $\P_q(\cC(t)) = \P_q(\cC(-t))$ for any $t$, we see that when $k$ is odd
\[
\sum_{t^2 \le 4q}  \P_q(\cC(t)) t^{k} = 0.
\]

Birch gives formulas for these types of expressions \cite[equation (4)]{Birch}.  For a nonnegative integer $R$, let 
\est{
\E_q(t_E^{2R}) = \frac{1}{q}\sum_{E\in \cC} \frac{t_E^{2R}}{\#\Aut_{\F_q}(E)} = \sum_{t^2\le 4q} \P_q(\cC(t)) t^{2R}.
} 
\begin{theorem}[Birch]\label{birch} 
For prime $p\geq 5$ we have 
\est{p\E_p(1) = & p \\ 
p\E_p(t_E^{2}) = & p^2-1 \\ 
p\E_p(t_E^{4}) = & 2p^3-3p-1 \\ 
p \E_p(t_E^{6}) = & 5p^4-9p^2-5p-1 \\ 
p\E_p(t_E^{8}) = & 14p^5-28p^3-20p^2 -7p-1 \\ 
p \E_p(t_E^{10}) = & 42p^6 -90p^4-75p^3 -35p^2-9p-1 -\tau(p),} 
where $\tau(p)$ is Ramanujan's $\tau$-function. 
\end{theorem}
Birch only gives such formulas for prime fields, but the extension to all finite fields is now well known and is implicit in work of Ihara \cite{Ihara}.  In the case where $q = p^v$ for $v\ge 2$, the $q \E_q(t_E^{10})$ formula also involves $\tau(q/p^2)$.  Formulas for $\E_q(t_E^{2R})$ involve traces of the Hecke operators $T_q$ and $T_{q/p^2}$ acting on spaces of cusp forms of weight at most $2R+2$ for $\SL_2(\Z)$.  See \cite[Theorem 2]{KaplanPetrow2} for a precise statement.

Applying these computations to the study of $W_{C_{2,3}^\perp}(X,Y)$ gives the following formulas.  
\begin{theorem}\label{LowWeightDual}
Let $\F_q$ be a finite field of characteristic not equal to $2$ or $3$.  We have
\begin{footnotesize}
\est{
W_{C_{2,3}^\perp}(X,Y) = X^{q^2+q+1} + q(q+1)(q-1)^2(q-2)(q^2+q+1)\Bigg[
 \frac{1}{5!} (q-3)X^{q^2+q-4} Y^5 &\\
 +   \frac{1}{6!} {(q-5)(q-4)(q-3)} X^{q^2+q-5} Y^6  
 +    \frac{1}{7!} (q-5)(q-4)(q-3)(q^2-6q+15) X^{q^2+q-6} Y^7 & \\
 +    \frac{1}{8!} (q-3)(2q^6 - 3q^5+79 q^4-797 q^3+2829 q^2-5110 q +4200) X^{q^2+q-7} Y^8 &\\
 +    \frac{1}{9!}\bigg(q^{10} + 3q^9 -16q^8 -585 q^7 +4262 q^6-7310 q^5-24393 q^4 +138512 q^3 &\\
 - 293174 q^2+333900 q -176400\bigg) X^{q^2+q-8} Y^9 \Bigg] + O(Y^{10}) &.
}
\end{footnotesize}
When $p \ge 5$ is prime, the $X^{p^2+p-9} Y^{10}$ coefficient of $W_{C_{2,3}^\perp}(X,Y)$ is
\begin{footnotesize}
\est{
\frac{p(p+1)(p-1)^2(p^2+p+1)}{10!} \Bigg(\bigg(p^{14}-43p^{12}+117p^{11}-2327 p^{10} +40444 p^9-287841 p^8 
+1088452 p^7  & \\
-2263884 p^6
 +1782811 p^5 
 +3312614 p^4 -  12006000 p^3 +17345160 p^2 -13807584 p +5080320\bigg) &\\
 - (p-1) p^2 \tau(p)\Bigg). &
}
\end{footnotesize}
\end{theorem}
\begin{remark}\label{dualremark}
When $q$ is a prime power the $X^{q^2+q-9}Y^{10}$ coefficient of $W_{C_{2,3}^\perp}(X,Y)$ also contains a term involving $\tau(q/p^2)$.  In general, the formula for the $X^{q^2+q+1-j} Y^j$ coefficient of $W_{C_{2,3}^\perp}(X,Y)$ involves traces of the Hecke operators $T_q$ and $T_{q/p^2}$ acting on spaces of cusp forms for $\SL_2(\Z)$ of weight at most $j+2$.  
\end{remark}

The coefficients computed in Theorem \ref{LowWeightDual} count something.  A weight $5$ codeword of $C_{2,3}^\perp$ corresponds to a tuple of points $p_1,\ldots, p_5 \in \Proj^2(\F_q)$ along with elements $a_1,\ldots, a_5 \in \F_q^*$ such that
\[
a_1 f(p'_1) + a_2 f(p'_2) + a_3 f(p'_3) + a_4 f(p'_4) + a_5 f(p'_5) = 0
\]
for all homogeneous cubics $f \in \F_q[x,y,z]_3$, where $p'_i$ denotes an affine representative of $p_i$.  It is not difficult to show that if $p_1,\ldots, p_5$ are not collinear, there is a cubic vanishing at $4$ of these points, but not all $5$.  In fact, the weight $5$ coefficient is exactly $q-1$ times the number of sets of $5$ collinear points in $\Proj^2(\F_q)$.  The points corresponding to a codeword of weight $8$ may be either collinear or lie on a smooth conic, and for larger weights more types of point configurations are possible.  In general, low-weight dual coefficients of $C_{n,k}^\perp$ count collections of points in $\Proj^n(\F_q)$ that fail to impose independent conditions on degree $k$ hypersurfaces.  The study of these collections is the subject of \emph{Interpolation Problems in Algebraic Geometry}.  For more information, see \cite{EisenbudGreenHarris, Harris}.

\section{Singular Affine Plane Cubic Curves}\label{SingularAffineCubics}

In this section we adapt our computation of $W_{C_{2,3}}^{\text{singular}}(X,Y)$ to compute $W_{C^A_{2,3}}^{\text{singular}}(X,Y)$.  We divide the singular affine cubics into two groups: those that are absolutely irreducible and those that are not.  We divide the set of reducible cubics into two further groups: those that contain an $\F_q$-rational affine line and those that do not.

\subsection{Absolutely Irreducible Singular Affine Plane Cubic Curves}\ \\

We give a result parallel to Proposition \ref{CountInflectionSmooth} for rational inflection lines of singular absolutely irreducible projective cubics.  As detailed in Lemma \ref{SingularProjectiveCubicList} there are three isomorphism classes of such cubics:
\begin{enumerate}[wide, labelwidth=!, labelindent=0pt]  
\item Cuspidal cubics, which have $q+1\ \F_q$-points;
\item Split nodal cubics, which have $q\ \F_q$-points;
\item Non-split nodal cubics, which have $q+2\ \F_q$-points.
\end{enumerate}

As in Section \ref{SmoothProjectiveCubic}, we write $\II(C)$ for the number of $\F_q$-rational inflection lines at nonsingular points of $C$.  In each case, the set of \emph{nonsingular} $\F_q$-rational points of $C$ forms a finite abelian group and $\II(C)$ depends on the structure of this group.  We assume that the characteristic of $\F_q$ is not $3$, because in characteristic $3$ all $\F_q$-points of the cuspidal cubic $\{zy^2 - x^3 = 0\}$ are inflection points.  Outside of this exceptional case, every absolutely irreducible singular plane cubic curve has exactly $3$ collinear inflection points over $\overline{\F}_q$ \cite[Theorem 5.1]{BruenHirschfeldWehlau}. When the group of nonsingular points has order divisible by $3$, there are two possibilities for $\II(C)$.  When this group has order not divisible by $3,\ \II(C) = 1$.  We summarize these statements below; see \cite{BruenHirschfeldWehlau} and \cite[Chapter 11]{Hirschfeld} for further discussion.
\begin{proposition}\label{sing_inflection}
Suppose $\F_q$ is a finite field of characteristic not equal to $3$.
\begin{enumerate}
\item If $q \equiv 1 \pmod{3}$ every cuspidal cubic and every non-split nodal cubic has $\II(C) =1$.  Of the $(q^3-1)(q^3-q)(q^3-q^2)/2$ homogeneous cubic polynomials defining split nodal cubics, $2/3$ have $\II(C) = 0$ and $1/3$ have $\II(C) = 3$.
\item If $q \equiv 2 \pmod{3}$ every cuspidal cubic and every split nodal cubic has $\II(C) =1$.  Of the $(q^3-1)(q^3-q)(q^3-q^2)/2$ homogeneous cubic polynomials defining non-split nodal cubics, $2/3$ have $\II(C) = 0$ and $1/3$ have $\II(C) = 3$.
\end{enumerate}
\end{proposition}

Let $C$ be a projective plane cubic curve.  We get an affine plane cubic by considering $C \setminus \{z=0\}$, the part of $C$ away from the \emph{line at infinity}.  The number of $\F_q$-points on this affine curve is 
\[
\#C(\F_q) - \#(C \cap \{z=0\})(\F_q).  
\]
Our strategy for computing the contribution to $W_{C_{2,3}^A}(X,Y)$ from a particular type of a cubic $C$ will be to determine the distribution of rational point counts for $C\cap L$ as we vary over all $\F_q$-rational lines $L$.  Since $\PGL_3(\F_q)$ acts transitively on lines, we can determine the contribution from curves isomorphic to $C$ as an average involving these point counts.

Let $C$ be a singular absolutely irreducible projective plane cubic curve with $\#C(\F_q) = q+1-t$.  Let $\LL_i(C)$ be the number of lines in $\Proj^2(\F_q)$ such that $\#(C\cap L)(\F_q) = i$. We have
\begin{equation}\label{LineCountA}
\LL_0(C) + \LL_1(C) + \LL_2(C) + \LL_3(C) = q^2+q+1.
\end{equation}
Every $\F_q$-rational line $L$ through the singular point of $C$ intersects $C$ in at most one other $\F_q$-point.  Exactly $q-t$ of these lines have $\#(C\cap L)(\F_q) = 2$.  The remaining $1+t$ of these lines have $\#(C\cap L)(\F_q) = 1$.   There are $q-t$ tangent lines at nonsingular points in $C(\F_q)$.  The $\II(C)\ \F_q$-rational inflection lines have $\#(C\cap L)(\F_q) = 1$, and the remaining $q-t-\II(C)$ tangent lines have $\#(C\cap L)(\F_q) = 2$.  Therefore, 
\begin{equation}\label{LineCountB}
\LL_2(C) = 2q-2t-\II(C).
\end{equation}
Every other line passing through $2$ distinct points of $C(\F_q)$ also passes through a third.  Therefore,
\begin{equation}\label{LineCountC}
\LL_3(C) = \frac{\binom{q-t}{2} - (q-t-\II(C))}{3}.
\end{equation}
Since there are $q+1\ \F_q$-rational lines passing through each point of $C(\F_q)$, we see that 
\begin{equation}\label{LineCountD}
\LL_1(C) + 2 \LL_2(C) + 3 \LL_3(C) = (q+1)(q+1-t).
\end{equation}
Equations \eqref{LineCountA}, \eqref{LineCountB}, \eqref{LineCountC}, and \eqref{LineCountD} determine $\LL_0(C), \LL_1(C), \LL_2(C)$, and $\LL_3(C)$.

\begin{proposition}
Let $\F_q$ be a finite field of characteristic not equal to $3$.  Let $W_{C_{2,3}^A}^{\text{sing. irred.}}(X,Y)$ denote the contribution to $W_{C_{2,3}^A}(X,Y)$ from singular affine cubics irreducible over $\overline{\F}_q$. 

Let 
\small{
\begin{eqnarray*}
W_{C_{2,3}^A}^{\text{cusp}}(X,Y) & =&  (q-1)(q^3-q)q^2 \Bigg( \frac{(q+1)(q-1)}{3} X^{q+1} Y^{q^2-q-1} + \frac{q^2-q+4}{2} X^q Y^{q^2-q} \\
& & + (2q-1) X^{q-1} Y^{q^2-q+1} + \frac{(q-1)(q-2)}{6} X^{q-2} Y^{q^2-q+2} \Bigg),\\
W_{C_{2,3}^A}^{\text{split}}(X,Y) & =&  \frac{(q-1)(q^3-q)(q^3-q^2)}{2} \Bigg( \frac{q(q+1)}{3} X^{q} Y^{q^2-q} + \frac{q^2-q+6}{2} X^{q-1} Y^{q^2-q+1} \\
& & + (2q-3) X^{q-2} Y^{q^2-q+2} + \frac{(q-2)(q-3)}{6} X^{q-3} Y^{q^2-q+3} \Bigg),
\end{eqnarray*}
}
and let
\small{
\begin{eqnarray*}
W_{C_{2,3}^A}^{\text{non-split}}(X,Y) & =&  \frac{(q-1)(q^3-q)(q^3-q^2)}{2} \Bigg( \frac{q(q-1)}{3} X^{q+2} Y^{q^2-q-2} \\
& & + \frac{q(q-1)}{2} X^{q+1} Y^{q^2-q-1}  + (2q+1) X^{q} Y^{q^2-q} + \frac{q(q-1)}{6} X^{q-1} Y^{q^2-q+1} \Bigg).
\end{eqnarray*}
}
We have 
\[
W_{C_{2,3}^A}^{\text{sing. irred.}}(X,Y) = W_{C_{2,3}^A}^{\text{cusp}}(X,Y)+ W_{C_{2,3}^A}^{\text{split}}(X,Y) + W_{C_{2,3}^A}^{\text{non-split}}(X,Y).
\]
\end{proposition}

\begin{proof}
Proposition \ref{sing_inflection} and the discussion of the $\LL_i(C)$ following it lead to the formula for $W_{C_{2,3}^A}^{\text{cusp}}(X,Y)$ for all $q\not\equiv 0 \pmod{3}$.  This also gives the formula for $W_{C_{2,3}^A}^{\text{split}}(X,Y)$ when $q \equiv 2\pmod{3}$.  To see that the same formula holds when $q \equiv 1\pmod{3}$, simplify the expression
\begin{small}
\begin{eqnarray*}
& & W_{C_{2,3}^A}^{\text{split}}(X,Y)  =  
\frac{(q-1)(q^3-q)(q^3-q^2)}{2} \Bigg(\left( \frac{2}{3} \cdot \frac{q^2+q+1}{3} + \frac{1}{3} \cdot \frac{(q+2)(q-1)}{3}\right) X^q Y^{q^2-q} \\
& & + \left(\frac{2}{3}  \cdot \frac{q^2-q+4}{2} + \frac{1}{3} \cdot \frac{q^2-q+10}{2}\right) X^{q-1}Y^{q^2-q+1} \\
& &   + \left( \frac{2}{3} \cdot (2q-2) + \frac{1}{3}\cdot  (2q-5)\right) X^{q-2} Y^{q^2-q+2} \\
&  & + \left(\frac{2}{3} \cdot \frac{(q-1)(q-4)}{6} + \frac{1}{3} \cdot \frac{q^2-5q +10}{6}\right) X^{q-3} Y^{q^2-q+3}\bigg).
\end{eqnarray*}
\end{small}
A similar simplification shows that the expression for $W_{C_{2,3}^A}^{\text{non-split}}(X,Y)$ holds when $q\equiv 1 \pmod{3}$ and when $q\equiv 2 \pmod{3}$. 
\end{proof}

\subsection{Reducible Affine Cubic Curves} \ \\

We divide the affine cubics reducible over $\overline{\F}_q$ into two groups: cubics that contain an $\F_q$-rational affine line and those that do not.  We begin with the second type.  There is a short list of such cubics.
\begin{enumerate}[wide, labelwidth=!, labelindent=0pt]  
\item \textbf{Three lines conjugate by $\Gal(\F_{q^3}/\F_q)$}:  These three lines can either be concurrent or not.  In the concurrent case, the $\F_q$-rational intersection point of these lines can either be in the affine plane, or on the line at infinity.  The contribution to $W_{C_{2,3}^A}(X,Y)$ from such cubics is
\[
\frac{(q-1)q^2 (q^3-q)}{3} X Y^{q^2-1} + \frac{(q-1)(q+1)(q^3-q)}{3} Y^{q^2} + \frac{(q-1)^2(q^5-q^3)}{3} Y^{q^2}.
\]

\item \textbf{Triple line at infinity}: Such cubics contribute $(q-1) Y^{q^2}$ to $W_{C_{2,3}^A}(X,Y)$.

\item \textbf{Line at infinity together with a smooth conic}: The contribution to $W_{C_{2,3}^A}(X,Y)$ from such cubics is
\[
(q-1)^2 (q+1) q^2 X^q Y^{q^2-q}+ \frac{(q-1)^2 q^3 (q+1)}{2} X^{q-1}Y^{q^2-q+1} +\frac{(q-1)^3 q^3}{2} X^{q+1} Y^{q^2-q-1}.
\]

\item \textbf{Line at infinity together with two lines conjugate by $\Gal(\F_{q^2}/\F_q)$}:  The $\F_q$-rational intersection point of these Galois-conjugate lines can either lie in the affine plane or on the line at infinity.  The contribution to $W_{C_{2,3}^A}(X,Y)$ from such cubics is
\[
(q-1)q^2 \frac{q^2-q}{2} XY^{q^2-1} +(q-1)(q+1) \frac{q^2-q}{2} Y^{q^2}.
\]

\end{enumerate}

We now analyze affine plane cubics that contain an $\F_q$-rational line.  We begin with the contribution to the weight enumerator from cubics that contain the particular line $\{x=0\}$.  The contribution to $W_{C_{2,3}^A}(X,Y)$ from such cubics is 
\begin{eqnarray*}
& & (q-1) \frac{2q^3-q^2- q+6}{2} X^q Y^{q^2-q} + \frac{q^2  (q - 1)^3}{2} X^{q+1}Y^{q^2-q-1} \\
& & + \frac{(q - 2) q^2 (q - 1)^3}{4} X^{2q-3} Y^{q^2-2 q+3} 
 + 2q^2(q -1)^3 X^{2q-2}Y^{q^2-2q+2}  \\
& &+ \frac{(q - 1) q^3 (q^2 - 2q + 7)}{2} X^{2q-1} Y^{q^2-2q+1}
  + (q - 1)^2  (q^3 - q^2 + 3) X^{2q} Y^{q^2-2q}   \\
& &  + \frac{(q - 2) q^2 (q - 1)^3}{4}X^{2 q+1} Y^{q^2-2q-1} 
  + \frac{q^2  (q - 1)^3}{2} X^{3q-3} Y^{q^2-3q+3} \\
& &   + 2 (q - 1)^2 q^2 X^{3q-2}Y^{q^2-3q+2} + \frac{(q - 2) (q - 1)^2}{2} X^{3q}Y^{q^2-3q}.
\end{eqnarray*}
\noindent  This computation is elementary but intricate, and we omit the details. 

The contribution to $W_{C_{2,3}^A}(X,Y)$ from cubics that contain at least one $\F_q$-rational affine line is the number of lines in the affine plane, $q^2+q$, times the polynomial above, minus the contribution from cubics that contain exactly $2\ \F_q$-rational lines, minus twice the contribution from cubics that contain exactly $3\ \F_q$-rational lines.  There are not many types of affine cubics that contain two or three $\F_q$-rational lines, and we omit this calculation.  This completes the determination of $W_{C_{2,3}^A}^{\text{singular}}(X,Y)$.

\section{Smooth Affine Plane Cubic Curves}\label{AffineSmooth}

We begin with an argument parallel to the one given in Section \ref{SingularAffineCubics} about intersections of $\F_q$-rational lines and irreducible singular cubic curves.  Let $C$ be a smooth projective plane cubic curve with $\#C(\F_q) = q+1-t$ and $\II(C)\ \F_q$-rational inflection lines.  Recall from Proposition \ref{CountInflectionSmooth} that for an elliptic curve $E$, the number of smooth cubics $C$ such that $C \cong E$ and  $\II(C)= k$ depends on $E(\F_q)[3]$.  Let $\LL_i(C)$ be the number of $\F_q$-rational lines $L$ with $\#(C\cap L)(\F_q) = i$.  We follow the argument given for singular cubics in Section \ref{SingularAffineCubics} and see that:
\begin{eqnarray*}
& & \LL_0(C) + \LL_1(C) + \LL_2(C) + \LL_3(C) = q^2+q+1, \\
& & \LL_1(C) + 2 \LL_2(C) + 3 \LL_3(C) = (q+1)(q+1-t) ,\\
& & \LL_2(C) = q+1-t - \II(C),\\
& & \LL_3(C) = \frac{1}{3} \left(\binom{q+1-t}{2} - \LL_2(C)\right).
\end{eqnarray*}
\noindent These equations determine $\LL_i(C)$ given $\II(C)$ and $t$.  

\begin{proposition}\label{CubicLinePairs}
Let $E$ be an elliptic curve defined over $\F_q$ with $\#E(\F_q) = q+1-t$.  The number of pairs $(C,L)$ such that $C$ is a projective plane cubic curve with $C  \cong E$ and $L$ an $\F_q$-rational line such that $\#(C \cap L)(\F_q) = i$ is $\#\PGL_3(\F_q)/\#\Aut_{\F_q}(E)$ times the polynomial given below:
\begin{align*}
& \frac{q^2 + qt + t^2 - q + t}{3}  & \text{ if } i = 0& ; & \frac{q^2 - t^2 + q + t + 2}{2}  &\ \ \ \text{    if } i = 1;&\\
& q-t  & \text{ if } i =2&; & \frac{(q-t)(q-t-1)}{6}  &\ \ \ \text{    if } i =3.&
\end{align*}

\end{proposition}

\begin{proof}
By Lemma \ref{NumberCubics}, there are $\#\PGL_3(\F_q)/\#\Aut_{\F_q}(E)$ cubic curves $C$ with $C  \cong E$.  If $E(\F_q)[3]$ is trivial, then every smooth cubic $C$ with $C  \cong E$ has $\II(C) = 1$.  If $\II(C) = 1$, 
\begin{align*}
& \LL_0(C)  =  \frac{q^2+qt+t^2-q+t}{3}, & 
& \LL_1(C)  =  \frac{q^2-t^2+q+t+2}{2}, &\\
& \LL_2(C)  =   q-t, &
& \LL_3(C)  =  \frac{(q-t)(q-t-1)}{6},&
\end{align*}
matching the polynomials given in the statement of the proposition. 

We can similarly solve for $\LL_i(C)$ for the other possible values of $\II(C)$.  
\begin{enumerate}
\item When $E(\F_q)[3] \cong \Z/3\Z$ Proposition \ref{CountInflectionSmooth} says that $1/3$ of all cubics $C$ with $C  \cong E$ have $\II(C) = 3$ and $2/3$ have $\II(C) = 0$.  Adding $1/3$ times the value of $\LL_i(C)$ for $\II(C) = 3$ and $2/3$ times the value of $\LL_i(C)$ for $\II(C) = 0$ gives the value of $\LL_i(C)$ stated in the proposition.  

\item When $E(\F_q)[3] \cong \Z/3\Z \times \Z/3\Z$ Proposition \ref{CountInflectionSmooth} says that $1/9$ of all cubics $C$ with $C  \cong E$ have $\II(C) = 9$ and $8/9$ have $\II(C) = 0$.  Adding $1/9$ times the value of $\LL_i(C)$ for $\II(C) = 9$ and $8/9$ times the value of $\LL_i(C)$ for $\II(C) = 0$ gives the value of $\LL_i(C)$ stated in the proposition.
\end{enumerate}

\end{proof}

\begin{theorem}\label{WE_C23A_Smooth}
Let $\F_q$ be a finite field of characteristic not equal to $3$. Let 
\begin{eqnarray*}
W^{\alpha}(X,Y,t)   = &
\frac{q^2 + qt + t^2 - q + t}{3} X^{q+1-t} Y^{q^2-q-1+t}  + \frac{q^2 - t^2 + q + t + 2}{2} X^{q-t} Y^{q^2-q+t}  \\
& +  (q-t) X^{q-1-t} Y^{q^2-q+1+t} + \frac{(q-t)(q-t-1)}{6} X^{q-2-t} Y^{q^2-q+2+t}.
\end{eqnarray*}
We have
\begin{eqnarray*}
W_{C^A_{2,3}}^{\text{smooth}}(X,Y) & = & q (q-1)(q^3-q)(q^3-q^2) \sum_{\substack{t^2 \le 4q \\ t \not\equiv 0 \pmod*{3}} }  \P_q(\cC(t))  W^{\alpha}(X,Y,t).
\end{eqnarray*}
\end{theorem}

\begin{proof}
Let $E$ be an elliptic curve defined over $\F_q$ with $\#E(\F_q)  = q+1-t$.  Proposition \ref{CubicLinePairs} gives the distribution of $\#C(\F_q) - \#(C\cap L)(\F_q)$ as we vary over all smooth cubic curves $C$ with $C  \cong E$ and all $\F_q$-rational lines $L$.  Since $\PGL_3(\F_q)$ acts transitively on $\F_q$-rational lines, this is $q^2+q+1$ times the distribution of values of $\#C(\F_q) - \#(C\cap \{z=0\})(\F_q)$.  There are $q-1$ nonzero polynomials defining each smooth cubic curve $C$.  Therefore, such an elliptic curve $E$ contributes
\[
(q-1) \frac{\#\PGL_3(\F_q)}{(q^2+q+1) \#\Aut_{\F_q}(E)} W^{\alpha}(X,Y,t) 
\]
to the weight enumerator $W_{C^A_{2,3}}^{\text{smooth}}(X,Y)$.  Varying over all $E\in \cC$ completes the proof.

\end{proof}

\section{Low-Weight Coefficients of $W_{(C^A_{2,3})^\perp}(X,Y)$}

We apply Theorem \ref{MacThm} to the expression for $W_{C_{2,3}^A}(X,Y)$ from Theorem \ref{WE_C23A_Smooth} in order to prove formulas analogous to those of Theorem \ref{LowWeightDual}.  As in the projective case, this leads to the expressions considered in Theorem \ref{birch}.

\begin{theorem}\label{LowWeightAffineDual}
Let $\F_q$ be a finite field of characteristic not equal to $2$ or $3$.  We have
\begin{small}
\est{
W_{(C^A_{2,3})^\perp}(X,Y) 
= 
X^{q^2} 
+(q-2) (q-1)^2 q^2 (q+1)  \\
\Bigg[ \frac{1}{5!} (q-4) (q-3)X^{q^2-5} Y^5 
 +   \frac{1}{6!} {(q-5)^2 (q-4) (q-3)} X^{q^2-6} Y^6  &\\
 +    \frac{1}{7!}  (q-6) (q-5) (q-4)(q-3) (q^2-6q+15) X^{q^2-7} Y^7 &\\
 +    \frac{1}{8!}  (q-3)(2q^7 -17q^6 + 121 q^5 -1161 q^4 + 7127 q^3 -23212 q^2 +     39340 q - 29400)  X^{q^2-8} Y^8 & \\
 +  \frac{1}{9!}  \bigg(q^{11} - 5q^{10}- 12q^9 -485 q^8 +8788 q^7 - 53642 q^6 + 142167 q^5 - 30540 q^4 & \\
 - 818744 q^3 +2249352 q^2 - 2731680 q + 1411200\bigg) 
 X^{q^2-9} Y^9\Bigg] & \\
 + O(Y^{10}).
}
\end{small}
 When $p \ge 5$ is prime, the $X^{p^2-10} Y^{10}$ coefficient of $W_{(C_{2,3}^A)^\perp}(X,Y)$ is
\begin{small}
\est{
\frac{(p-1)^2 p^2 (p+1) }{10!}\bigg(\big(p^{15}-9p^{14}-7p^{13}+384 p^{12} -4514 p^{11} +68191 p^{10}-706065 p^9 & \\
+4482991 p^8 - 18172206 p^7 +
 47512147 p^6 -75728017 p^5  +54600840 p^4  + 36872568 p^3  &\\
 -125756064 p^2  +120294720 p  -45722880\big)  & \\
- p(p-1) (p^2-9p+36) \tau(p)
 \bigg). &
}
\end{small}
\end{theorem}
\begin{remark}
\begin{enumerate}[wide, labelwidth=!, labelindent=0pt]  
\item As in the discussion following Theorem \ref{LowWeightDual}, weight $k$ codewords in $(C_{2,3}^A)^{\perp}$ come from special configurations of $k$ points in affine space failing to impose independent conditions on homogeneous cubic polynomials.  For example, the weight $5$ coefficient of $W_{(C^A_{2,3})^\perp}(X,Y)$ is $q-1$ times the number of collections of $5$ collinear points in the affine space $\F_q^2$.  

\item Also as in the discussion following Theorem \ref{LowWeightDual}, when $q$ is a prime power, the $X^{q^2-10} Y^{10}$ coefficient of $W_{(C_{2,3}^A)^\perp}(X,Y)$ includes a term involving $\tau(q/p^2)$, but we do not compute it here.  In general, the weight $k$ coefficeint of $W_{(C_{2,3}^A)^\perp}(X,Y)$ will involve traces of the Hecke operators $T_q$ and $T_{q/p^2}$ acting on spaces of cusp forms for $\SL_2(\Z)$ of weight at most $k+2$.
\end{enumerate}
\end{remark}

\subsection{The weight enumerator of the dual code and curves with prescribed $3$-torsion}\ \\

Proposition \ref{CountInflectionSmooth} demonstrates how elliptic curves with prescribed $3$-torsion play a role in enumerative questions about affine plane cubic curves.  In Proposition \ref{CubicLinePairs} we saw that when we average over all cubic curves $C$  isomorphic to a particular elliptic curve $E$, the effect of $E(\F_q)[3]$ disappears.  We do not need to consider cubic curves in groups based on the $3$-torsion of the associated elliptic curve, but recent work of the author and Petrow shows that we could divide things in this way and still obtain explicit formulas.

We consider set of smooth projective plane cubic curves $C$ based on $E(\F_q)[3]$ for the elliptic curve $E$ satisfying $C  \cong E$.  By inclusion-exclusion, we can do this by dividing these cubics into those for which $E(\F_q)[3]$ has a subgroup isomorphic to $\Z/3\Z$ and those for which $E(\F_q)[3] \cong \Z/3\Z \times \Z/3\Z$.  Following the terminology from \cite{KaplanPetrow2}, let $\cC(A_{3,3} t)$ denote  the set of isomorphism classes of elliptic curves $E \in \cC$ with $E(\F_q)[3] \cong \Z/3\Z \times \Z/3\Z$ and $\#E(\F_q) = q+1-t$.  The following result is a special case of a weighted version of \cite[Theorem 4.9]{Schoof}. 
\begin{lemma}\label{Schoof2}
Suppose that $p$ is the characteristic of $\F_q$ and that $t\in \Z$ satisfies $t^2 \le 4q$. Then
\begin{alignat*}{3}
\P_q(\cC(A_{3,3} t)) = &\frac{1}{2q}H\left(\frac{t^2 - 4q}{9}\right)\quad && \text{ if } q\equiv 1 \pmod{3},\ p\nmid t, \text{ and }t\equiv q+1 \pmod*{9};\\
 = & \P_q(\cC(2\sqrt{q})) \quad && \text{ if } q \text{ is a square } p\neq 3,\ t = 2 \sqrt{q}, \text{ and } \sqrt{q}\equiv 1 \pmod*{3}; \\
 = & \P_q(\cC(-2\sqrt{q})) \quad && \text{ if } q \text{ is a square } p\neq 3,\ t = -2 \sqrt{q}, \text{ and } \sqrt{q}\equiv -1 \pmod*{3}; \\
 =& 0 \quad && \text{ otherwise}.
\end{alignat*} 
\end{lemma}

Applying the MacWilliams theorem, using the binomial theorem to isolate a particular coefficient of $W_{(C_{2,3}^A)^{\perp}}(X,Y)$ leads to  expressions of the following type:
\begin{eqnarray*}
\E_q(t_E^R \Phi_{\Z/3\Z}) & \coloneqq & \sum_{\substack{t^2 \le 4q \\ q+1 - t \equiv 0\pmod*{3}}}  q\P_q(\cC(t)) t^R;\\
\E_q(t_E^R \Phi_{\Z/3\Z \times \Z/3\Z}) & \coloneqq & \sum_{\substack{t^2 \le 4q \\ q+1 - t \equiv 0\pmod*{9}}}  q\P_q(\cC(A_{3,3} t)) t^R.
\end{eqnarray*}

The author and Petrow give formulas for exactly these types of expressions in Theorem 3 of \cite{KaplanPetrow2}.  Stating the full result would require introducing too much additional notation, so we refer to \cite{KaplanPetrow2} for details.  The case where $R$ is even and $q = p$ is prime is addressed in \cite[Examples 1 and 2]{KaplanPetrow2}.  We see that $\E_p(t_E^{2R} \Phi_{\Z/3\Z})$ can be expressed in terms of the traces of the Hecke operator $T_p$ acting on spaces of cusp forms for $\Gamma_1(3)$ of weight at most $2R+2$, and that $\E_p(t_E^{2R} \Phi_{\Z/3\Z \times \Z/3\Z})$ can be expressed in terms of traces of $T_p$ acting on spaces of cusp forms for $\Gamma(3)$ of weight at most $2R+2$.  Formulas for these quantities when $q$ is a prime power, or when $R$ is odd, are more intricate. See \cite[Theorem 3]{KaplanPetrow2} for details.  

We give an example of a formula we get from applying the MacWilliams theorem to a subset of smooth projective plane cubic curves that satisfy additional constraints on the $3$-torsion of the associated elliptic curve. 
\begin{theorem}
\begin{enumerate}[wide, labelwidth=!, labelindent=0pt]  
\item Let $q$ be a prime with $q\equiv 1 \pmod{3}$.  The $X^{q^2+q-1} Y^2$ coefficient of
\[
\sum_{\substack{t^2 \le 4q \\ q+1-t \equiv 0 \pmod*{9}}} \P_q(\cC(A_{3,3} t)) (X+(q-1)Y)^{q+1-t} (X-Y)^{q^2+t} 
\]
is
\[
\frac{-(q+1) (q-1)^3 q^4 (q^2+q+1)}{48} \bigg((7q+3) + q \rm{Tr}(T_q | S_4(\Gamma(3)))\bigg),
\]
where $\rm{Tr}(T_q | S_4(\Gamma(3)))$ denotes the trace of the $T_q$ Hecke operator acting on the space of holomorphic weight $4$ cusp forms for $\Gamma(3)$.

\item Let $q$ be a prime with $q \equiv 2\pmod{3}$.  The $X^{q^2+q-3} Y^4$ coefficient of 
\[
\sum_{\substack{t^2 \le 4q \\ q+1-t \equiv 0 \pmod*{3}}} \P_q(\cC(t)) (X+(q-1)Y)^{q+1-t} (X-Y)^{q^2+t} 
\]
is  \small{
\[
\frac{-(q+1)(q-1)^3 q^4 (q^2+q+1)}{48}  \bigg( (q+1) (q^5-7q^4+20q^3-26q^2+13q +2) +
   q^3 \rm{Tr}(T_q | S_6(\Gamma_0(3)))\bigg),
\]
}
where $\rm{Tr}(T_q | S_6(\Gamma_0(3)))$ denotes the trace of the $T_q$ Hecke operator acting on the space of holomorphic weight $6$ cusp forms for $\Gamma_0(3)$.
\end{enumerate}

\end{theorem}

\section{Acknowledgements}

Part of this project grew out of the PhD thesis of the author. He thanks Noam Elkies for his extensive guidance and for many helpful conversations.  The author thanks the referee for several very helpful suggestions.  He also thanks Joseph Gunther for helpful discussions.  The author is supported by NSA Young Investigator Grant H98230-16-10305.

\end{document}